\documentclass[11pt]{amsart} 
\usepackage{amsthm}
\usepackage{amsmath}
\usepackage{amssymb}
\usepackage{extarrows}
\usepackage{enumerate}
\usepackage[margin=2.9cm]{geometry}
\usepackage{comment}
\usepackage{xcolor}

\usepackage[doi=true,isbn=true,giveninits=true,maxbibnames=99,backend=biber]{biblatex}
\usepackage{hyperref}
\usepackage{xurl}
\hypersetup{colorlinks=true,breaklinks=true,linktoc=all,urlcolor={blue!80!black},citecolor={green!60!black}}
\newtheorem{theorem}{Theorem}[section] 
\newtheorem{lem}[theorem]{Lemma}
\newtheorem{prop}[theorem]{Proposition}
\newtheorem{cor}[theorem]{Corollary}
\newtheorem*{theorem*}{Theorem}
\newtheorem*{lem*}{Lemma}
\newtheorem*{prop*}{Proposition}
\newtheorem*{cor*}{Corollary}

\theoremstyle{definition}
\newtheorem*{notat*}{Notations}
\newtheorem*{quest*}{Question}

\newtheorem{rem}[theorem]{Remark}
\newtheorem{rems}[theorem]{Remarks}
\newtheorem*{rem*}{Remark}
\newtheorem*{rems*}{Remarks}
\theoremstyle{definition}
\newtheorem{ex}[theorem]{Example}
\newtheorem*{ex*}{Example}
\newtheorem{defin}[theorem]{Definition}
\newtheorem*{defin*}{Definition}

\numberwithin{equation}{section}

\def \H{{\mathbb{H}}}

\def \T{{\mathbb{T}}}

\def \C{{\mathbb{C}}}

\def \D{{\mathbb{D}}}

\def \R{{\mathbb{R}}}

	\def \cG{{\mathcal{G}}}	\def \cM{{\mathcal{M}}}		
\def \cB{{\mathcal{B}}}	\def \cH{{\mathcal{H}}}			
			\def \cU{{\mathcal{U}}}
\def \cD{{\mathcal{D}}}			\def \cV{{\mathcal{V}}}
			\def \cW{{\mathcal{W}}}
\def \cF{{\mathcal{F}}}			

\def \xx{\times}
\def \bd{\partial}
\DeclareMathOperator{\supp}{supp}
\DeclareMathOperator{\dist}{dist}

\renewcommand{\Im}{\operatorname{Im}}
\newcommand{\clos}[1]{\bar{#1}}
\let\comp\circ 
\newcommand{\0}[1]{\overline{#1}}
\newcommand{\2}[1]{{}_{|#1}}
\newcommand{\wvv}[1]{\widehat{#1}}

\newcommand{\inline}[1]{\quad\text{#1}\quad}
\newcommand{\afterline}[1]{\qquad\text{#1}\ }
\def \var{\,\cdot\,}
\def \such{\; : \;}

\NewDocumentCommand{\parder}{mo}{%
	\IfNoValueTF{#2}
	{\frac{\partial}{\partial{#1}}}
	{\frac{\partial^{#2}}{\partial{#1}^{#2}}}
}

\NewDocumentCommand{\fparder}{mom}{%
	\IfNoValueTF{#2}
	{\frac{\partial{#3}}{\partial{#1}}}
	{\frac{\partial^{#2}{#3}}{\partial{#1}^{#2}}}
}

\let\oldin\in 
\DeclareRobustCommand{\in}{\oldin\nolinebreak[4]}
\let\indentedrule\hrulefill 
\renewcommand{\hrulefill}{\noindent\indentedrule}

\DeclareMathOperator{\const}{const}

\begin{document}

\title{Free boundary problems via Sakai's theorem}

\author{D.~Vardakis}
\address{Department of Mathematics, Michigan State University, East Lansing, MI. 48823}
\email{jimvardakis@gmail.com}

\author{A.~Volberg}
\address{Department of Mathematics, Michigan State University, East Lansing, MI. 48823;\hfill\break\indent
		Hausdorff Center for Mathematics, Bonn, Germany}
\email{volberg@msu.edu}
\dedicatory{To Nikolai Nikolski, who taught me what is important in mathematics and in life\break
	 --- Sasha Volberg}
\thanks{The second author was partially supported by NSF grant DMS 1900286 and Alexander von Humboldt Foundation}




\begin{abstract}
	A Schwarz function on an open domain $\Omega$ is a holomorphic function satisfying $S(\zeta)=\overline{\zeta}$ on $\Gamma$, which is part of the boundary of $\Omega$. Sakai in 1991 gave a complete characterization of the boundary of a domain admitting a Schwarz function. In fact, if $\Omega$ is simply connected and $\Gamma=\partial \Omega\cap D(\zeta,r)$, then $\Gamma$ has to be regular real analytic. 
	This paper is an attempt to describe $\Gamma$ when the boundary condition is slightly relaxed. 
	In particular, three different scenarios over a simply connected domain $\Omega$ are treated: when $f_1(\zeta)=\overline{\zeta}f_2(\zeta)$ on $\Gamma$ with $f_1,f_2$ holomorphic and continuous up to the boundary, when $\mathcal{U}/\mathcal{V}$ equals certain real analytic function on $\Gamma$ with $\mathcal{U},\mathcal{V}$ positive and harmonic on $\Omega$ and vanishing on $\Gamma$, and when $S(\zeta)=\Phi(\zeta,\overline{\zeta})$ on $\Gamma$ with $\Phi$ a holomorphic function of two variables. It turns out that the boundary piece $\Gamma$ can be, respectively, anything from $C^\infty$
	to merely $C^1$, regular except finitely many points, or regular except for a measure zero set.
\end{abstract}

\maketitle

\section{Introduction} \label{sec_intro} 
Let $D(\zeta_0, r)\subset{\C}$ denote the open disk centred at $\zeta_0\in{\C}$ and of radius $r>0$. Let $\Omega$ be an open subset of $D(\zeta_0,r)$ where $\zeta_0\in \Gamma=\bd \Omega\cap D(\zeta_0,r)$ is a non-isolated boundary point.

 A Schwarz function of $\Omega\cup \Gamma$ is a function $S:\Omega\cup\Gamma\to{\C}$ holomorphic on $\Omega$ and continuous on $\Omega\cup \Gamma$ that satisfies
\begin{equation} \label{eqSchwarz}
	S(\zeta)=\0{\zeta}\afterline{on}\Gamma.
\end{equation}

In his Acta Mathematica paper \cite{Sak1991}, Sakai proved that Schwarz functions completely characterize the shape of $\Gamma$. One of the technical tools used was the Phragm{\'e}n--Lindel{\"o}f principle in the form below, but it is far from being the key to his proof; his paper is full of very subtle tricks.

\begin{theorem} \label{Fuchs}
	Let $\Omega$ be an open set in ${\C}$ and let $\zeta_0$ be a non-isolated boundary point of $\Omega$. Let $f$ be a holomorphic function on $\Omega$ and $D(\zeta_0,\delta)$ a ball satisfying the following:
	\begin{enumerate}[{\rm(i)}]
		\item $\limsup|f(z)|\leq 1$ while $\Omega\ni z\to \zeta$ for every $\zeta\in\bd \Omega\cap D(\zeta_0,\delta)\setminus\{\zeta_0\}$ and
		\item $|f(z)|\leq \alpha|z-\zeta_0|^{-\beta}$ in $\Omega\cap D(\zeta_0,\delta)$ for some positive constants $\alpha$ and $\beta$.
	\end{enumerate}
	Then,
	\[
	\limsup|f(z)|\leq 1
	\]
	while $\Omega\ni z\to \zeta_0$.
\end{theorem}

In particular, Sakai proved the following, see \cite[Theorem 5.2]{Sak1991}.
\begin{theorem} \label{SakTheorem}
	Let $\Omega\subset D(\zeta_0,r)$	be a bounded open set in ${\C}$ and $\zeta_0$ an non-isolated point of its boundary, $\Gamma=\bd \Omega\cap D(\zeta_0,r)$. Suppose $S$ is a Schwarz function on $\Omega\cup \Gamma,$ that is,
	\begin{enumerate}[{\rm(i)}]
		\item $S$ is holomorphic on $\Omega,$
		\item continuous on $\Omega\cup \Gamma,$ and
		\item $S(\zeta)=\0{\zeta}$ on $\Gamma$. \label{Schwarz_condition}
	\end{enumerate}
	Then, for some small $0<\delta\leq r$ one of the following must occur (where we set $D=D(\zeta_0,\delta)$):
	\begin{enumerate}[{\rm(2a)}]
		\item[{\rm(1)}] $\Omega\cap D$ is simply connected and $\Gamma\cap D$ is a regular real analytic simple arc through $\zeta_0;$
		\item $\Gamma\cap D$ determines uniquely a regular real analytic arc through $\zeta_0;$ $\Gamma\cap D$ is either an infinite proper subset of this arc with $\zeta_0$ as an accumulation point or equal to it; also, $\Omega\cap D=D\setminus \Gamma;$
		\item $\Omega\cap D=\Omega_1\cup \Omega_2$ where $\Omega_1$ and $\Omega_2$ are (open) simply connected and $\bd \Omega_1\cap D$ and $\bd \Omega_2\cap D$ are regular real analytic simple arcs through $\zeta_0$ and tangent at $\zeta_0;$
		\item $\Omega\cap D$ is simply connected and $\Gamma\cap D$ is a regular real analytic simple arc except for a cusp at $\zeta_0;$ the cusp points into $\Omega$.
	\end{enumerate}
\end{theorem}
\noindent Recall that a \emph{regular arc} means a differentiable arc whose derivative never vanishes and \emph{simple} means that it is parametrized by an injective continuous function.

\begin{rems} \label{Sak_cusp_case}
	 Here is an example of a cusp of (2c) at $\zeta_0=0$ with Schwarz function. There exist analytic functions $T$ on $\{|z|\leq \eta\}$, for some $\eta>0$, that have a zero of order $2$ at $0$, are univalent on closed upper half-disk ${K_\eta\equiv\big\{|z|\leq \eta\such\Im(z)\geq0\big\}}$,	and satisfy $\Gamma\cap D\subset T(-\eta,\eta)$ and $T(K_\eta)\subset \Omega\cup \Gamma$. In fact, it is easy to construct such functions. Every such $T$ leads to a Schwarz function on the domain ${\Omega=T(\{|z|< \eta,\ \Im z>0\})}$, which has two analytic arcs forming a cusp $\Gamma$ at~$0$. In order to have $S(\zeta)=\bar \zeta$ on $\Gamma$, it suffices to have a function analytic in $\{|z|< \eta, \Im z>0\}$ and continuous up to $(-\eta, \eta)$ such that ${A(x)=\overline{T(x)}}$, $x\in (-\eta, \eta)$. Having such an $A$ we set $S=A\circ T^{-1}$ on~$\Omega$. On the other hand, using that $T$ is analytic in the whole ball $\{|z|\leq \eta\}$, we can choose $A$ as follows: $A(z)= \overline{T(\bar z)}$. Moreover, Sakai \cite{Sak1991} showed that every Schwarz function on a cusp domain appears because of an analytic function~$T$ as above.
	 
	The converse of this theorem also holds, in the sense that if any of the conditions (1), (2a), (2b), or (2c) is satisfied, then $\Omega$ admits a Schwarz function.
\end{rems}

In order to distinguish between the cases, Sakai also showed an auxiliary result \cite[Proposition 5.1]{Sak1991}, which we will also use here.
\begin{theorem} \label{SakProposition}
	Set $D'=D(0,r)$. Let $\Omega'\subset D'$ be an open set and $0$ an accumulation point of its boundary, $\Gamma'=\bd \Omega'\cap D'$. Then, for some $r'\leq r,$ either
	\begin{enumerate}[{\rm(1)}]
		\item there exists a Schwarz function, $S_t,$ of $(\Omega'\cup \Gamma')\cap D(0,r')$ at $0$ if and only if there exists a function $\Phi_1$ defined on $(\Omega'\cup \Gamma')\cap D(0,\delta)$ for some $\delta>0$ such that
		\begin{enumerate}[{\rm(i)}]
			\item $\Phi_1$ is holomorphic and univalent in $\Omega'\cap D(0,\delta),$
			\item $\Phi_1$ is continuous on $(\Omega'\cup \Gamma')\cap D(0,\delta),$
			\item $\Phi_1(\zeta)=|\zeta|^2$ on $\Gamma'\cap D(0,\delta)$
		\end{enumerate}
		or
		\item there exists a Schwarz function, $S_t,$ of $(\Omega'\cup \Gamma')\cap D(0,r')$ at $0$ if and only if there exists a function $\Phi_2$ defined on $(\Omega'\cup \Gamma')\cap D(0,\delta)$ for some $\delta>0$ such that
		\begin{enumerate}[{\rm(i')}]
			\item $\Phi_2$ is holomorphic and univalent in $\Omega'\cap D(0,\delta),$
			\item $\Phi_2^2$ is continuous on $(\Omega'\cup \Gamma')\cap D(0,\delta),$
			\item $\Phi_2^2(\zeta)=|\zeta|^2$ on $\Gamma'\cap D(0,\delta),$
			\item $\Phi_2(\Omega'\cap D(0,\delta))\cup (-\epsilon,\epsilon)$ contains a neighbourhood of $0$ for $\epsilon>0$.
		\end{enumerate}
	\end{enumerate}
	In particular, the functions $\Phi_1,\Phi_2$ are related to $S_t$ by $\Phi_1(z)=zS_t(z)$ and $\Phi_2(z)=\sqrt{zS_t(z)}$.
\end{theorem}

Unfortunately, Theorem~\ref{SakProposition} is only valid around $0$ in this form. Nevertheless, we can ``translate'' the setup of Theorem~\ref{SakTheorem} by setting $\Omega'=\Omega-\zeta_0$, $\Gamma'=\Gamma-\zeta_0$ and $S_t(z)=S(z+\zeta_0)-\0{\zeta}_0$ for $z\in \Omega'$. Then, $S_t$ is a Schwarz function on $\Omega'\cup \Gamma'$ at $0$. Cases (1) of the two theorems correspond with one another as do (2a), (2b), and (2c) with (2).

Sakai gave two applications of his results: the first one describes the local structure of the boundary of quadrature domains, while the second one deals with a free boundary problem of classical type, namely, what is the boundary of the set of positivity of a smooth \emph{non-negative} function in the disk such that $\Delta u=1$ on the set $\{u>0\}$.

\smallskip

It is natural to wonder how one can derive similar results for other forms of~\eqref{eqSchwarz}. In this text, we examine three different scenarios for a simply connec\-ted domain $\Omega$. In \S\S\ref{sec_polynomials}, \ref{sec_model_spaces}, \ref{sec_construction} equation \eqref{eqSchwarz} is replaced by
\begin{equation} \label{eqIntroSchwarz}
	f_1(\zeta)=\0{\zeta}f_2(\zeta)\afterline{for all}\zeta\in\bd \Omega
\end{equation}
where $f_1,f_2$ are holomorphic functions continuous up to the boundary. This is closely related to the model subspaces $K_\theta$ and Nevanlinna domains, which will be important here. It is shown that there are domains so that \eqref{eqIntroSchwarz} holds for which $\bd \Omega$ is $C^\infty$ but not real analytic. Further, in \S\ref{sec_phi} we replace the quantity $\0{\zeta}f_2(\zeta)$ with $\Phi(\zeta,\0{\zeta})$, where $\Phi$ is a holomorphic function of two variables, to find that the boundary is locally composed of real analytic arcs. Finally, in \S\ref{sec_u_v} we consider two positive harmonic functions $\cU$ and $\cV$ that are zero on a Jordan arc, $\Gamma$, of the boundary. If their ratio on $\Gamma$ is equal to a real analytic function of the form $|A|^2$, where $A$ is holomorphic, then $\Gamma$ is real analytic itself with the possible exception of some cusps.

Our interests to the problems considered below also was spurred by an application, which originates from complex dynamics. A certain complex dynamics question naturally brought the second author to another \emph{free boundary problem} described in \S\ref{sec_u_v}. After that it was very natural to ask related questions, where the Sakai setup was generalized in yet two other ways. To our surprise the answers were quite different and required different techniques: from the use of Nevanlinna domains and pseudo-continuation to multivalued analytic functions.

\section{Polynomials \& analytic functions} \label{sec_polynomials} 

Let $\Omega$ be an open domain, $\zeta_0$ a non-isolated boundary point of $\Omega$, and let $\Gamma=\bd \Omega\cap D(\zeta_0,r)$ for some $r>0$. Suppose $S$ is a holomorphic function on~$\Omega$ continuous on $\Omega\cup \Gamma$. We start with a simple yet important case. Instead of~\eqref{eqSchwarz}, we consider
\begin{equation} \label{eqPolySchwarz}
	S(\zeta)=\0{\zeta} p(\zeta)\afterline{on}\Gamma,
\end{equation}
where $p$ is a polynomial. We will shortly show that $f(z)=\frac{S(z)}{p(z)}$ is, in fact, a Schwarz function on $\Gamma$.

\begin{lem} \label{translate_to_zero}
	Assume that $S:\Omega\to{\C}$ is holomorphic on $\Omega\subset D(\zeta_0,r),$ continuous on $\Omega\cup \Gamma,$ and that it satisfies
	\[
	S(\zeta)=\0{\zeta} (\zeta-\zeta_0)^n\afterline{on}\Gamma.
	\]
	Then, the function $S_t(z)=S(z+\zeta_0)-\0{\zeta_0}z^n$ is holomorphic on $\Omega-\zeta_0\subset D(0,r),$ continuous on $(\Omega-\zeta_0)\cup (\Gamma-\zeta_0)$ and it satisfies
	\[
	S_t(\zeta)=\0{\zeta} \zeta^n\afterline{on}\Gamma-\zeta_0.
	\]
\end{lem}

\begin{prop} \label{monomial_schwarz}
	Assume $0\in \Gamma$ is a non-isolated boundary point of ${\Omega\subset D(0,r)}$ and suppose $S$ is a holomorphic function on $\Omega$ continuous on $\Omega\cup \Gamma$ and satisfying
	\[
	S(\zeta)=\0{\zeta} \zeta^n\afterline{on}\Gamma.
	\]
	Then, for any positive $\delta<r$ the function $\frac{S(z)}{z^n}$ is holomorphic on $\Omega\cap D(0,\delta)$ and continuous on $(\Omega\cup \Gamma)\cap D(0,\delta)\setminus\{0\}$. Moreover, the following holds while $z\in \Omega\cup \Gamma\setminus\{0\}$:
	\[
	\lim_{z\to 0} \frac{S(z)}{z^n}=0.
	\]
\end{prop}

\begin{proof}
	The function $\frac{S(z)}{z^n}$ is clearly holomorphic on $\Omega\cap D(0,\delta)$ and continuous on $(\Omega\cup \Gamma)\cap D(0,\delta)\setminus\{0\}$ for any $\delta\in(0,r)$. It remains to see what happens at~$0$.
	
	Fix $\delta\in(0,r)$. Since $S$ is bounded on $\Omega\cap D(0,r)$, say by $m$, we get
	\[\left|\frac{S(z)}{z^n}\right|\leq m|z|^{-n}\afterline{on}\Omega\cap D(0,\delta)\]
	and additionally for any $\zeta\in \Gamma\cap D(0,\delta)\setminus\{0\}$ we have
	\[\lim\left|\frac{S(z)}{z^n}\right|=|\0{\zeta}|\leq \delta\afterline{while}\Omega\ni z\to \zeta.\]
	Hence, by the Phragm\'en-Lindel{\"o}f principle \ref{Fuchs} we obtain
	\[\limsup\left|\frac{S(z)}{z^n}\right|\leq \delta\afterline{while}\Omega\ni z\to 0.\]
	This last inequality is true for any positive $\delta<r$ and therefore $\lim\frac{S(z)}{z^n}=0$ as $z\to 0$.
\end{proof}

\begin{cor} \label{polynomial_Schwarz}
	Let $p$ be a complex polynomial. Assume that $\zeta_0\in \Gamma$ is a non-isolated boundary point of $\Omega$ other than zero and suppose $S$ is a holomorphic function of $\Omega\subset D(\zeta_0,r)$ continuous on $\Omega\cup \Gamma$ and satisfying
	\[S(\zeta)=\0{\zeta} p(\zeta)\afterline{on}\Gamma.\]
	Set $f(z)=S(z)/p(z)$ on $\Omega\cup \Gamma\setminus\{\zeta_0\}$ and $f(\zeta_0)=\0{\zeta_0}$. Then, $f$ is a Schwarz function of $\Omega\cup \Gamma$ on $D(\zeta_0,r)$ for sufficiently small $r>0$.
\end{cor}

\begin{proof}
	Take $r$ so small that $p$ has no zeros on $\0{D(\zeta_0,r)}\setminus\{\zeta_0\}$. If $p(\zeta_0)\neq0$, the result is immediate.
	
	If $p(\zeta_0)=0$, we only need to show that $f$ is continuous on $(\Omega\cup \Gamma)\cap D(\zeta_0,r)$. Denote by $n$ the order of $\zeta_0$ as a zero of $p$ and consider the function 
	\[S_n(z)=S(z)\frac{(z-\zeta_0)^n}{p(z)}.\]
	$S_n$ is holomorphic on $\Omega$, is continuous on $\Omega\cup \Gamma$, and satisfies
	\[S_n(\zeta)=\0{\zeta}(\zeta-\zeta_0)^n\afterline{on}\Gamma.\]
	From Lemma \ref{translate_to_zero} we get 
	\[(S_n)_t(\zeta)=\0{\zeta}\zeta^n\afterline{on}\Gamma-\zeta_0\]
	and from Proposition \ref{monomial_schwarz} we deduce that while 
	\begin{align*}
		\lim_{z\to 0}\frac{(S_n)_t(z)}{z^n}=0 & \implies \lim_{z\to 0}\frac{S_n(z+\zeta_0)-\0{\zeta_0}z^n}{z^n}=0\\
		& \implies \lim_{z\to \zeta_0}\frac{S_n(z)-\0{\zeta_0}(z-\zeta_0)^n}{(z-\zeta_0)^n}=0\\
		& \implies\lim_{z\to \zeta_0}f(z)=\lim_{z\to \zeta_0}\frac{S_n(z)}{(z-\zeta_0)^n}=\0{\zeta_0}
	\end{align*}
$z\in \Omega$,	and the conclusion follows.
\end{proof}

Notice that the same proof works with $p$ replaced by any function $F$ that is analytic in a neighbourhood of $\zeta_0$. This along with Lemma~\ref{translate_to_zero} give us the following corollary.

\begin{cor} \label{analytic_schwarz}
	Assume $\zeta_0\in \Gamma$ is a non-isolated boundary point of ${\Omega\subset D(\zeta_0,r)}$. Suppose $F$ is a function analytic around $\zeta_0$ and $S$ is a holomorphic function on~$\Omega$ continuous on $\Omega\cup \Gamma$ and satisfying
	\[S(\zeta)=\0{\zeta} F(\zeta)\afterline{on}\Gamma.\]
	Set $f(z)=S(z)/F(z)$ on $(\Omega\cup \Gamma)\cap D(\zeta_0,\delta)\setminus\{\zeta_0\}$ for some sufficiently small $\delta>0$ and $f(\zeta_0)=\0{\zeta_0}$. Then, $f$ is a Schwarz function of $\Omega\cup \Gamma$ on $D(\zeta_0,\delta)$.
\end{cor}

The converse of this corollary also holds true in the sense that if $\Gamma$ has certain shape, in particular, if it satisfies {(1)}, {(2a)}, {(2b)}, or {(2c)} of \ref{SakTheorem}, then there is a Schwarz function $f$ of $\Omega\cup \Gamma$ at $\zeta_0$ such that $S(\zeta)=\0{\zeta} F(\zeta)$ on $\Gamma$ where $S=Ff$.

In fact, we can slightly modify the same proof to get a little more, again through the Phragm\'en-Lindel{\"o}f principle \ref{Fuchs}.

\begin{cor}
	Let $p$ be a polynomial, $F$ a function analytic in a neighbourhood of $\clos{\Omega},$ and $S$ a function holomorphic on the (bounded) set $\Omega$ and continuous on $\Omega\cup \Gamma$. Suppose that for all $\zeta\in \Gamma$ we have 
	\[S(\zeta)=p(\0{\zeta})F(\zeta).\]
	Then, for every non-isolated point $\zeta_0$ of the boundary $\Gamma$ for which $p'(\zeta_0)\not=0,$ there is some $\delta>0$ such that the function $p^{-1}(S/F)$ is a Schwarz function of $\Omega\cup \Gamma$ on $D(\zeta_0,\delta)$.
\end{cor}

We wish to examine what happens in the more general case where $p$ in \eqref{eqPolySchwarz} is replaced with any analytic function of $\Omega$ continuous on its boundary, but not necessarily analytic on that boundary. More specifically, suppose that $f_1$ and $f_2$ are functions analytic on $\Omega$, continuous on $\Omega\cup \Gamma$, and satisfying
\begin{equation} \label{eqAnalScharz}
	f_1(\zeta)=\0{\zeta}f_2(\zeta)\afterline{on}\Gamma.
\end{equation}
As above, if $f_2(\zeta_0)\neq0$, the function $f=f_1/f_2$ is a Schwarz function around $\zeta_0\in \Gamma$ and no issues arise. However, if $f_2(\zeta_0)=0$, the situation is very complicated in general.

We start with a lemma analogous to Lemma~\ref{translate_to_zero}:
\begin{lem} \label{move_to_zero}
	Assume that $f_1,f_2:\Omega\to{\C}$ are holomorphic on $\Omega\subset D(\zeta_0,r),$ continuous on $\Omega\cup \Gamma,$ and that they satisfy
	\[f_1(\zeta)=\0{\zeta}f_2(\zeta)\afterline{on}\Gamma.\]
	Then, there exist functions $(f_1)_t$ and $(f_2)_t$ holomorphic on $\Omega-\zeta_0,$ continuous on $(\Omega-\zeta_0)\cup(\Gamma-\zeta_0)$ and such that
	\[(f_1)_t(\zeta)=\0{\zeta}(f_2)_t(\zeta)\afterline{on}\Gamma-\zeta_0.\]
	If additionally $f_2(\zeta_0)=0,$ then $(f_2)_t(0)=0$.
\end{lem}

\begin{proof}
	Define $(f_1)_t$ by
	\[(f_1)_t(z)=f_1(z+\zeta_0)-\0{\zeta_0}f_2(z+\zeta_0).\]
	Then for $\zeta\in \Gamma-\zeta_0$ we have
	\begin{align*}
		(f_1)_t(\zeta) & =f_1(\zeta+\zeta_0)-\0{\zeta_0}f_2(\zeta+\zeta_0)\\
		& =\0{\zeta+\zeta_0}f_2(\zeta+\zeta_0)-\0{\zeta_0}f_2(\zeta+\zeta_0)\\
		& =\0{\zeta}f_2(\zeta+\zeta_0)
	\end{align*}
	Setting $(f_2)_t(z)=f_2(z+\zeta_0)$, we have the desired identity.
	
	Clearly, $(f_1)_t(0)=0$ and also if $f_2(\zeta_0)=0$, $(f_2)_t(0)=0$.
\end{proof}
\noindent Abusing the notation, we denote these new functions again by $f_1$ and $f_2$.

\smallskip

It remains to show a result analogous to Corollary~\ref{polynomial_Schwarz} with $p$ replaced by $f_2$. In particular, we would like to show that the function $f=f_1/f_2$ is holomorphic on $\Omega$, continuous on $\Gamma$, and that it satisfies
\[f(\zeta)=\frac{f_1(\zeta)}{f_2(\zeta)}=\0{\zeta}\afterline{for all}\zeta\in \Gamma.\]
However, the limit of $f_1(z)/f_2(z)$ as $\Omega\ni z\to 0$ may even fail to exist when $f_2(0)=0$, and we cannot apply the Phragm{\'e}n-Lindel{\"o}f principle here. We will need to see this problem from a different scope.

\section{Nevanlinna domains and inner functions} \label{sec_model_spaces} 

We recall that a bounded simply connected domain $\Omega$ is called a Nevanlinna domain if there exist bounded holomorphic functions $f_1$, $f_2$ in $\Omega$ such that 
\[\overline{\varphi(z)} = \frac{f_1(\varphi(z))}{f_2(\varphi(z))}\]
for almost every $z\in \T=\{z: |z|=1\}$, where $\varphi$ is a conformal mapping of the unit disk onto $\Omega$. Note that this definition does not imply any additional regularity (for instance, continuity) of the functions $f_1,f_2$ on $\bd\Omega$.

We will restrict the above situation, and suppose there are holomorphic functions $f_1,f_2:\Omega\to{\C}$ continuous up to the boundary that satisfy
\begin{equation} \label{eq_small_letters}
	f_1(\zeta)=\0{\zeta}f_2(\zeta)\afterline{for}\zeta\in \Gamma.
\end{equation}

In order to better understand the situation, we rewrite \eqref{eq_small_letters} as
\begin{equation} \tag{3.1'}
\label{eq_small_letters_fraction}
	\frac{f_1(\zeta)}{f_2(\zeta)}=\0{\zeta},
\end{equation}
which is now fulfilled almost everywhere on $\Gamma$ except for the closed set ${\Gamma\cap f_2^{-1}\{0\}}$, which has zero measure. Then, $\Omega$ is what we call a \emph{strong Nevanlinna domain} and if such $f_1$ and $f_2$ exist, the ratio $f_1/f_2$ is unique thanks to the Lusin-Privalov uniqueness theorem.

Let $\phi:\D\to \Omega$ be a conformal map and consider the functions $F_1=f_1\comp \phi$ and $F_2=f_2\comp \phi$. Formulas \eqref{eq_small_letters} and \eqref{eq_small_letters_fraction} transform respectively to
\begin{equation} \label{eq_big_letters}
	F_1(\zeta)=\0{\phi(\zeta)}F_2(\zeta)
\end{equation}
and
\begin{equation} \tag{3.2'} 
\label{eq_big_letters_fraction}
	\frac{F_1(\zeta)}{F_2(\zeta)}=\0{\phi(\zeta)}
\end{equation}
both of which hold true in the sense of angular boundary values almost everywhere on~$\T$, because~$\phi$ may fail to extend ``nicely'' to~$\clos \D$. By the factorization theorem, 
we can write $F_1$ and $F_2$ in $\D$ as
\begin{equation} \label{eq_big_letters_factorized}
	F_1=\theta_1\cF_1\inline{and}F_2=\theta_2\cF_2
\end{equation}
where the $\cF_i$ are the outer factors of $F_i$ and the~$\theta_i$ are their inner factors. Since $F_1,F_2\in H^\infty$, also $\cF_1,\cF_2\in H^\infty$, and from \eqref{eq_big_letters_fraction} we get
\begin{equation} \label{eq_big_letters_fraction_factorized}
	\frac{\theta_1(\zeta)}{\theta_2(\zeta)}\frac{\cF_1(\zeta)}{\cF_2(\zeta)}=\0{\phi(\zeta)},
\end{equation}
almost everywhere on $\T$ in the sense of angular boundary values. We distinguish between two cases: either $\theta_2$ divides $\theta_1$, that is, $\theta_1/\theta_2\in H^\infty$, or it does not.

\subsection{\texorpdfstring{$\theta_2\mid \theta_1$}{?2 divides ?1}.} 
Let $h=\theta_1/\theta_2\in H^\infty$. Then, the function $(h\cF_1)/\cF_2$ belongs to the class $N^+$, defined as
\[N^+=\bigg\{\frac{f}{g}\such f,g\in H^\infty,\ g\text{ is an outer function}\bigg\},\]
and its (angular) boundary values are equal almost everywhere on $\T$ to the (angular) boundary values of $\0{\phi}$. However, since $\Omega$ is bounded, we see that $\phi\in L^\infty(\T,m)$ where $m$ is the normalized Lebesgue measure on $\T$. Smirnov's Theorem 
tells us that in fact $(h\cF_1)/\cF_2\in H^\infty$. Therefore, we have a bounded holomorphic function on the disk that is equal to $\0{\phi}$ almost everywhere on $\T$. This is impossible whenever $\phi$ is a bounded holomorphic function on $\D$. 

We are necessarily left with the other case.

\subsection{\texorpdfstring{$\theta_2\nmid \theta_1$}{?2 doesn't divide ?1}.}
We begin with some notation and definition which will be important for the rest of this text.

Let $\D_e=\wvv{\C}\setminus\clos\D$. For any function $h:\D\to{\C}$ we define $\widetilde h$ as
\[\widetilde h(z)=\0{h(1/\0 z)}.\]
The notation $\widetilde H$ will stand for a function $\widetilde H:\D_e\to{\C}$ and we will write $H$ instead of $\widetilde {\widetilde H}$ for the function $\0{\widetilde H(1/\0 z)}$. Observe that $h\in H^\infty$ if and only if $\widetilde h\in H^\infty(\D_e)$, and $h(0)=0$ if and only if $\widetilde h(\infty)=0$.

We will also consider the backward shift operator, $\cB: H^p\to H^p$, for $p\in[1,\infty)$, that is
\[\cB: f\mapsto\frac{f(z)-f(0)}{z}.\]

\begin{defin}
	Let $f$ be a meromorphic function on $\D$. We say that~$f$ admits \emph{pseudo-continuation} (across $\T$) if there exists another meromorphic function $g$ on $\D_e$ such that $f=g$ almost everywhere (on $\T$) in the sense of non-tangential limits.
	
	The pseudo-continuation of $f$ is called \emph{of bounded type} or a \emph{Nevanlinna-type pseudo-continua-tion} if $g$ is of the form $g=h_1/h_2$ for some $h_1,h_2\in H^\infty(\D_e)$.
\end{defin}

\begin{defin}
	A function $f\in H^p$ is called a \emph{cyclic vector for $\cB$}, or simply \emph{cyclic for $\cB$} if the set $\{\cB^nf\}_{n=0}^\infty$ spans the space $H^p$.
\end{defin}

The following important result is due to Douglas, Shapiro, and Shields.

\begin{theorem} \label{non_cyclic_functions}
	Consider $1\leq p<\infty$. A function $f\in H^p$ is not cyclic for $\cB$ if and only if $f$ has a pseudo-continuation of bounded type.
\end{theorem}

In the case when $p=2$, it is known that any non-cyclic function of $\cB$ belongs to a proper $\cB$-invariant subspace.
As a consequence of Beurling's theorem, these spaces are of the form $(\theta H^2)^\perp$ and are known as \emph{model spaces} and denoted by~$K_\theta$. Here we will need the fact that
\[K_\theta=(\theta H^2)^\perp=H^2(\T)\cap \theta\0{H^2_0(\T)},\]
where in the last identity we mean the boundary values of the corresponding functions and where $H^2_0=\{f\in H^2\such f(0)=0\}$.

Now, we can proceed with the case when $\theta_2\nmid \theta_1$:

After dividing both $\theta_1$ and $\theta_2$ by their greatest common divisor, we may assume that $\theta_1$ and $\theta_2$ have no common zeros and that the Borel supports of their singular measures are disjoint. Much as above, we see that the function $F=(\theta_1\cF_1)/\cF_2=F_1/\cF_2$ belongs the class $N^+$ and thus $F\in H^\infty$, because $\theta_2\0{\phi}\in L^\infty(\T,m)$. Then the following is true in the sense of angular boundary values for almost every $\zeta\in\T$:
\begin{align}
	&	\0{\phi(\zeta)}=\frac{\theta_1(\zeta)\cF_1(\zeta)}{\theta_2(\zeta)\cF_2(\zeta)}=\frac{F(\zeta)}{\theta_2(\zeta)}\notag\\
	\iff & \phi(\zeta)=\theta_2(\zeta)\0 F(\zeta) \label{eq-is-in-model-space}\\
	\iff & \phi(\zeta)=\frac{\widetilde F(\zeta)}{\widetilde {\theta}_2(\zeta)}. \label{eq_found_theta}
\end{align}
Since $\widetilde F,\widetilde {\theta}_2\in H^\infty(\D_e)$, we see that $\phi\in H^\infty\subset H^2$ admits pseudo-continuation across $\T$ of bounded type, and Theorem~\ref{non_cyclic_functions} shows that $\phi$ is not cyclic for $\cB$. So, it has to belong to some model space $K_\theta$. See \cite[Theorem 1]{Fed2006} for more details. In fact, from~\eqref{eq-is-in-model-space} and because we ``need'' to have $F(0)=0$, it follows that either
\[\phi\in K_{\theta_2} \text{ if } \theta_1(0)=0,\inline{or}\phi\in K_{z\theta_2} \text{ if } \theta_1(0)\neq0.\]

\section{Boundary behaviour of conformal maps in \texorpdfstring{$K_\theta$}{model spaces}} \label{sec_construction} 

In this section we show that Theorem~\ref{SakTheorem} fails when condition \eqref{Schwarz_condition} is replaced by \eqref{eq_small_letters}. To this end, we will find a simply connected domain $\Omega$ and a conformal map $\phi: \D\to \Omega$ continuous up the boundary that has a pseudo-continuation of bounded type and is smooth but not real analytic on $\T$. The functions participating in this pseudo-continuation will also be continuous on the boundary. First, we go one step back and work with Nevanlinna domains. Thanks to \cite[Theorem 1]{Fed2006} by Fedorovskiy, this is equivalent to studying the model subspaces, $K_\theta$, for different inner functions $\theta$. 

If $\theta(z_0)=0$ for some $z_0\in\D$, the function
\[\phi(z)=\frac{1}{1-\0{z_0}z}\in K_\theta\cap C^\infty(\T)\]
has bounded type pseudo-continuation across $\T$ and thus $\phi(\D)$ is a Nevanlinna domain. In fact, $\phi$ can be analytically extended on the whole closed disk, $\clos{\D}$, and~$\phi(\T)$ is real analytic. On the other hand, in a series of papers, \cite{Maz1997,CarParFed2002,Fed2006,Maz2016,Maz2018,BelBorFed2019}, it has been shown that the boundary of a Nevanlinna domain can be ``arbitrarily bad''. In particular, it can be nowhere analytic \cite{Maz1997}, of class $C^1$ but not in any $C^{1,\alpha}$ for no $\alpha>0$ \cite{Fed2006}, or even non-rectifiable \cite{Maz2016}. We refer also to the Belov-Fedorovskiy paper \cite{BelFed2018}, where the description is given of model spaces that contain bounded univalent functions. We mention that the Hausdorff dimension of the accessible boundary of a Nevanlinna domain can be any number between $1$ and $2$ as shown in \cite{BelBorFed2019}, another construction can be found in \cite{Maz2018}.

\smallskip

However, in all the above work the inner function $\theta$ is a Blaschke product or has a Blaschke part. Moreover, in order to compare with Sakai's theorem, we have to consider the case where the functions $\widetilde F_1,\widetilde F_2\in H^\infty(\D_e)$ for which $\phi=\widetilde F_1/\widetilde F_2$ on $\T$ are continuous up to $\T$. This is not always possible when $\theta$ is not purely singular (see \cite[Example 5.8]{CarParFed2002}).

Therefore, in this section $\theta$ will be a singular inner function of the form
\[\theta(z)=\exp\bigg(-\int\limits_\T\frac{\zeta+z}{\zeta-z}d\mu_\theta(\zeta)\bigg)\]
with $\mu_\theta$ supported on a Carleson set, $E\subset\T$. We will show that there is a conformal map $\phi\in K_\theta$ continuous on $\clos{\D}$ which is in $C^\infty(\T)$ but not real analytic on $\T$.

In view of \cite[Theorem 2.1]{DyaKha2006}, since $\supp(\mu_\theta)$ is Carleson, the space $K_\theta$ then contains a non-trivial function from some smoothness class, for example a function $g\in H^\infty\cap C^\infty(\T)$ (or in a Bergman space, i.e., $g\in A^{p,1}$ for some $p>1$). Since $g\in K_\theta$, it admits a bounded type pseudo-continuation of the form
\[g=\widetilde G/\widetilde {\theta}\afterline{almost everywhere on}\T,\]
where $\widetilde G\in H^\infty(\D_e)$ vanishes at infinity (see \cite[Theorem 5.1.4]{CimRos2000}). Additionally,~$g$ has an analytic continuation, say $\cG$, to $\wvv{\C}\setminus\supp(\mu)$. Of course, $\cG=\widetilde G/\widetilde {\theta}$ on $\D_e$ and observe that $\cG$ cannot be bounded in $\D_e$; otherwise $g$ would be constant, as $\cG\2\D=g$ and $\cG\2{\D_e}$ coincide almost everywhere on $\T$.

Now, consider $\alpha\in\D_e$ with $\theta(1/\0{\alpha})\neq0$ and the following aggregate:
\[\phi(z)=\frac{\cG(z)-\cG(\alpha)}{z-\alpha}.\]
We will show that $\phi\in K_\theta\cap C^\infty(\T)$ and $\phi$ is conformal in $\clos\D$.

Clearly, $\phi$ is inside $H^2(\D)$ and also
\[\0{\theta(\zeta)}\phi(\zeta)=\frac{\0{\theta(\zeta)}g(\zeta)-\0{\theta(\zeta)}\cG(\alpha)}{\zeta-\alpha}=\frac{\widetilde G(\zeta)-\widetilde {\theta}(\zeta)\cG(\alpha)}{\zeta-\alpha}.\]
For $z\in\D_e$ the function
\[\frac{\widetilde G(z)-\widetilde {\theta}(z)\cG(\alpha)}{z-\alpha}=\frac{1}{z-\alpha}\bigg(\widetilde G(z)-\frac{\widetilde {\theta}(z)}{\widetilde {\theta}(\alpha)}\widetilde G(\alpha)\bigg)\]
is analytic around $\alpha$ and vanishes at infinity. Hence, $\phi\in K_\theta$.

Furthermore, $\phi$ is univalent in $\clos\D$. Indeed, suppose it is not. Then, there exist $z,w\in\clos\D$ with $z\neq w$ and $\phi(z)=\phi(w)$ or equivalently
\begin{align*}
	& \frac{g(z)-\cG(\alpha)}{z-\alpha}=\frac{g(w)-\cG(\alpha)}{w-\alpha}\\
	\iff & \frac{g(z)}{z-\alpha}-\frac{g(w)}{w-\alpha}=\frac{\cG(\alpha)}{z-\alpha}-\frac{\cG(\alpha)}{w-\alpha}=\cG(\alpha)\frac{z-w}{(z-\alpha)(w-\alpha)}\\
	\iff & -\alpha\frac{g(z)-g(w)}{z-w}+w\frac{g(z)-g(w)}{z-w}-g(w)=\cG(\alpha).
\end{align*}
The left-hand side is bounded, because $g\in C^\infty(\clos\D)$, whereas we can pick ${1<\alpha<2}$ so that $|\cG(\alpha)|$ is arbitrarily large (recall $\cG\2{\D_e}$ is not bounded), a contradiction, and therefore $\phi$ is univalent in $\clos\D$.

Consequently, if $g\in K_\theta\cap C^\infty(\T)$ and $\theta$ is a singular inner function, then~$\phi$ is univalent in $\clos\D$ and $\phi\in K_\theta\cap C^\infty(\T)$. Also see \cite[Section 4]{BarFed2011} for more details. At the same time, note that $\cG$ cannot be analytically extended to the whole~$\clos\D$, because it is unbounded near the unit circle, and thus neither can $\phi$; this fails exactly on the Carleson set $E$.

\smallskip

Now, since $\phi\in K_\theta$, we can write 
\begin{equation} \label{eq_conformal_in_K_theta}
	\phi=\theta\0 F\iff \theta\0{\phi}=F
\end{equation}
almost everywhere on $\T$ for some function $F\in H^2$ with $F(0)=0$. In fact, $F\in H^\infty$ because $\phi\in C^\infty(\clos\D)$.

It is known that there exists some analytic function, $\cH$, with $\cH\2 E=0$ such that both $\cH$ and $\cH \theta$ are Lipschitz on $\clos\D$. In fact, we can further consider $\cH$ to be an outer function in $C^\infty(\T)$.
Multiplying by $\cH$ in \eqref{eq_conformal_in_K_theta}, we get
\begin{equation} \label{multi_by_shirokov}
	(\cH \theta)\0{\phi}=\cH F
\end{equation}
almost everywhere on $\T$. In particular, the left-hand side is now smooth on the whole $\T$ and the same therefore holds true for the right-hand side. In a sense,~$\cH$ ``annihilates'' the singularities of $\theta$ as \eqref{eq_conformal_in_K_theta} fails exactly on the support,~$E$, of~$\mu_\theta$.

At this point, set $F_1=\cH F$, $F_2=\cH \theta$, and $f_j=F_j\comp \phi^{-1}$ for $j=1,2$. Then,~\eqref{multi_by_shirokov} becomes
\[F_1=\0{\phi}F_2,\]
which now is fulfilled on the whole boundary $\T$, and in turn
\[f_1(\zeta)=\0{\zeta}f_2(\zeta)\inline{for all}\zeta\in \Gamma.\]
This is exactly the setup we were looking for, albeit it contrasts with Sakai's result: Even though $\Gamma=\phi(\T)$ is $C^\infty$-smooth, $\phi$ cannot be analytic on the Carleson set $E$ and thus neither can $\Gamma$.

\smallskip

It is worth mentioning that there are examples of Nevanlinna domains that come from singular inner functions with particularly irregular boundaries. Namely, in \cite{BelBorFed2019} one can find examples of univalent functions in a Paley-Wiener space such that they map the upper half-plane onto a Nevanlinna domain whose boundary can have any dimension between $1$ and $2$.

\section{Holomorphic functions in \texorpdfstring{${\C}^2$}{C2}} \label{sec_phi} 

In this section we attempt to replace the function $\0{\zeta}f_0(\zeta)$ in \eqref{eq_small_letters} with a more general formula.

For some positive $r>0$, let $\Omega\subset D(\zeta_0,r)$ be a simply connected open set, let $\Gamma=\bd \Omega\cap D(\zeta_0,r)$, and let $\zeta_0\in \Gamma$. Here, we will also need the extra assumption that $\Gamma$ is a Jordan arc (or possibly a union of Jordan arcs).

Let $\Phi$ be a holomorphic function of two variables, that is, a function of the form
\[\Phi(z,w)=\sum_{n,m=0}^{+\infty}b_{nm}z^nw^m\]
where each of the functions $\Phi(z,\var)$ and $\Phi(\var,w)$ is itself holomorphic. Suppose there exists a function $R$ which is
\begin{enumerate}[(i)]
	\item holomorphic on $\Omega$,
	\item continuous on $\clos{\Omega}$, and \label{continuity_of_R}
	\item satisfies $R(\zeta)=\Phi(\zeta,\0{\zeta})$ on $\Gamma$. \label{eq_Sakai_doubly}
\end{enumerate}
In view of Lemma~\ref{move_to_zero}, we may assume that $\zeta_0=0$ and $b_{00}=0$ so that $R(0)=\Phi(0,0)=0$. Notice that $R(z)$ and $\Phi(z,\0 z)$ are bounded on $\clos{\Omega}$ and thanks to the Phragm\'en-Lindel{\"o}f Principle~\ref{Fuchs}, we may assume without loss of generality that there exists some non-negative integer $k$ for which
\begin{equation}\label{eq_weier_conditions}
	\Phi(0,0)=\parder{w}\Phi(0,0)=\dots=\parder{w}[k-1]\Phi(0,0)=0\inline{and}\parder{w}[k]\Phi(0,0)\neq0
\end{equation}
otherwise $\Phi$ would be identically zero.

We would like to use the Weierstra{\ss} approximation theorem for the function $\Phi(z,w)-R(z)$ around $0$, but $R$ is not holomorphic on the boundary. But since it is continuous by \eqref{continuity_of_R} and $\Gamma$ is Jordan, we can use Mergelyan's theorem to get a sequence of polynomials $p_n$ that converge to $R$ uniformly on $\clos{\Omega}$. And we can pick this sequence so that $p_n(0)=0$ for every $n=0,1,\dots$\,.

Next, we define the functions
\[\Psi(z,w)=\Phi(z,w)-R(z)\inline{and}\Psi_n(z,w)=\Phi(z,w)-p_n(z).\]
The $\Psi_n$ are holomorphic on ${\C}^2$ and converge uniformly to $\Psi$ on $\clos{\Omega}\xx{\C}$. Observe that for all $n$ we have $\Psi_n(0,0)=\Phi(0,0)-p_n(0)=0$ and also 
\[\parder{w}[\kappa]\Psi_n=\parder{w}[\kappa]\Phi\afterline{for all integers}\kappa\geq1\]
and all points $(z,w)$. Then, from \eqref{eq_weier_conditions} and from the Weierstra{\ss} approximation theorem, there exist unique holomorphic functions $a_{0;n},\dots,a_{k-1;n}:{\C}\to{\C}$ and $c_n:{\C}^2\to{\C}$ with $a_{j;n}(0)=0$ and $c_n(0,0)\neq0$ such that
\begin{equation*}
	\Psi_n(z,w)=c_n(z,w)\left(w^k+a_{k-1;n}(z)w^{k-1}+\dots+a_{0;n}(z)\right).
\end{equation*}

Following the proof of the Weierstra{\ss} theorem and since the convergence $\Psi_n\to \Psi$ is uniform on $\clos{\Omega}\xx{\C}$, we can find sufficiently small $\delta$ and $\rho$ with $\rho\geq \delta>0$ so that $a_{0;n},\dots,a_{k-1;n}$ and the $c_n$ converge uniformly on $\clos{\Omega}\cap D(0,\delta)$ and $\left(\clos{\Omega}\cap D(0,\delta)\right)\xx D(0,\rho)$, respectively, to some functions $a_0,\dots,a_{k-1}$ and $c$ with $a_j(0)=0$ and $c(0,0)\neq0$. Note that the functions $a_j$ are holomorphic on $\Omega\cap D(0,\delta)$ and continuous on $\clos{\Omega}\cap D(0,\delta)$. Subsequently, we get
\begin{equation} \label{eq_boundary_factorization}
	\Phi(z,w)-R(z)=c(z,w)\left(w^k+a_{k-1}(z)w^{k-1}+\dots+a_0(z)\right).
\end{equation}

Let us write
\[P(z,w)=w^k+a_{k-1}(z)w^{k-1}+\dots+a_0(z)\]
for the polynomial factor. From \eqref{eq_Sakai_doubly}, \eqref{eq_boundary_factorization} and since $c(0,0)\neq0$, we have 
\begin{equation} \label{eq_poly_zero_on_boundary}
	P(\zeta,\0{\zeta})=\0{\zeta}^k+a_{k-1}(\zeta)\0{\zeta}^{k-1}+\dots+a_0(\zeta)=0\afterline{for all}\zeta\in \Gamma\cap D(0,\delta).
\end{equation}

\begin{rem*}
	Functions of the form
	\[P(z,\0 z)=\0 z^k+a_{k-1}(z)\0 z^{k-1}+\dots+a_0(z),\]
	where $a_j$ are polynomials, are called polyanalytic polynomials. One can find more details on these in \cite{Fed1996,Maz2016} or \cite{Sie1988}.
\end{rem*}

We are interested in the roots of the polynomial $P(z,\var)$ when $z\in \clos{\Omega}\cap D(0,\delta)$. In other words, we will study the equation (in~$w$)
\begin{equation} \label{eq_zero_poly_on_boundary}
\begin{split}
	P(z,w)=0\iff w^k+a_{k-1}(z)w^{k-1}+\dots+a_0(z)=0\\
	\text{when }z\in \clos{\Omega}\cap D(0,\delta).
\end{split}
\end{equation}

Let $\cD(z)$ be the discriminant of $P(z,\var)$ (for any fixed $z$). Then, $\cD(z)$ is a polynomial of the coefficients $a_0(z),\dots,a_{k-1}(z)$ and is equal to $0$ if, and only if, $P(z,w)$ and $\parder{w}{P(z,w)}$ share a common factor. The roots of $P(z,\var)$ are given by a multivalued holomorphic function, $\cW$, depending on $a_1,\dots,a_{k-1}$, and the points where $\cW$ changes a branch inside $\Omega\cap D(0,\delta)$ are exactly the zeros of $\cD$ (in $\Omega\cap D(0,\delta)$).

\medskip

We distinguish between two cases: when $\cD$ is identically $0$ and when it is not.

Before moving on, let us note that the set $\cM(\Omega,\Gamma,\delta)$ of all meromorphic functions on $\Omega\cap D(0,\delta)$ continuous up to $(\Omega\cup \Gamma)\cap D(0,\delta)$ except possibly a (closed) measure zero subset of $\Gamma$ is a field with the usual operations of addition and multiplication.

\subsection{\texorpdfstring{$\cD\neq0$}{Non-Zero Discriminant}}

Here $P(z,w)$ is irreducible over $\cM(\Omega,\Gamma,\delta)$. Since $\cD$ is continuous on $(\Omega\cup \Gamma)\cap D(0,\delta)$, the set $\left(\cD^{-1}\{0\}\cap \Gamma\right)\cap D(0,\delta)$ is closed and of zero harmonic measure. Now, we decompose $\left(\Gamma\setminus\cD^{-1}\{0\}\right)\cap D(0,\delta)$ into countably many open connected arcs.

Let $\gamma$ be one of these arcs. Then, there exists a simply connected set $D\subset \Omega\cap D(0,\delta)$ such that $\bd D\cap\bd \Omega=\gamma$. Since $\cD$ has no zeros on $D\cup \gamma$, by the monodromy theorem the multivalued function $\cW$ ``splits'' into $k$ distinct holomorphic functions, $W_j$ ($j=1,\dots,k$), and let $C_j=\{\zeta\in \gamma\such W_j(\zeta)=\0{\zeta}\}$. Notice that the $C_j$'s are closed (in $\gamma$), they cover $\gamma$, and any two of them intersect at a (closed) set of zero harmonic measure.

Unfortunately, $C_j$ need not be connected, but we can further decompose each~$\mathring{C_j}$ (whenever it is non-empty) into countably many open arcs as in $\mathring{C_j}=\cup_i \gamma_j^i$, for $j=1,\dots,k$. Again, around each $\gamma_j^i$ we consider a neighbourhood $D_j^i\subset D$ with $\bd D_j^i\cap \bd D=\gamma_j^i$ (these can, but need not be simply connected) and let $W_j^i=W_j\2{D_j^i\cup \gamma_j^i}$.

Then, for each $j=1,\dots,k$ and $i=1,2\dots$ the functions $W_j^i$ are holomorphic on $D_j^i$, continuous on $D_j^i\cup \gamma_j^i$ and satisfy $W_j^i(\zeta)=\0{\zeta}$ for all $\zeta\in \gamma_j^i$; in other words, they are Schwarz functions on $D_j^i\cup \gamma_j^i$. Since $\Gamma$ is Jordan, all $\gamma_j^i$ are also Jordan and from Theorem~\ref{SakTheorem} we conclude that each $\gamma_j^i$ is, in fact, a regular real analytic simple arc except possibly some cusps.

\subsection{\texorpdfstring{$\cD=0$}{Zero Discriminant}}

In this case, $P(z,w)$ has to be reducible over $\cM(\Omega,\Gamma,\delta)$. In particular, we can write $P(z,w)=P_1(z,w)\cdots P_{\widetilde k}(z,w)$ for some $\widetilde k\leq k$ where each $P_\kappa(z,w)$ has now coefficients in $\cM(\Omega,\Gamma,\delta)$ and is irreducible, i.e., $\cD_\kappa\neq0$ where $\cD_\kappa$ is the discriminant of $P_\kappa(z,\var)$. Since $P(\zeta,\0{\zeta})=0$ for all $\zeta\in \Gamma$, we can split $(\Gamma\setminus E)\cap D(0,\delta)$, where $E$ is some closed zero-(harmonic)-measure set, into open sets $O_\kappa$ for $\kappa=1,\dots,\widetilde k$ so that $P_\kappa(\zeta,\0{\zeta})=0$ for all $\zeta\in O_\kappa$. Notice that $O_\kappa\cap O_{\kappa'}=\emptyset$ when $P_\kappa$ and $P_{\kappa'}$ are different.

Observe that, since $P(z,w)$ factors into the polynomials $P_\kappa(z,w)$ (over $\cM(\Omega,\Gamma,\delta)$) and the roots of $P(z,\var)$ are given by the multivalued holomorphic function $\cW$, the roots of each $P_\kappa(z,\var)$ are also given by a multivalued holomorphic function $\cW_\kappa$ whose branches are comprised of branches of $\cW$.

Working as above for each $\kappa=1,\dots,\widetilde k$, we separate $O_\kappa\setminus\cD_\kappa^{-1}\{0\}$ into countably many open arcs and for each such arc, $\gamma$, we find some simply connected neighbourhood, $D\subset \Omega$, with $\bd D\cap\bd \Omega=\gamma$ so that $\cW_\kappa$ ``splits'' into its different branches. Again following the above arguments, we can decompose $\gamma$ --- minus a zero-measure set --- into countably many open arcs over which $W_j(\zeta)=\0{\zeta}$ for some branch $W_j$ of $\cW_\kappa$. Constructing appropriate neighbourhoods, we conclude that except a zero-measure set, $\gamma$ is a countable union of regular real analytic simple arcs except possibly some cusps.

\smallskip

In either case, the cusps (if they exist) point into $\Omega$ and may only accumulate on the endpoints of each open arc.

\smallskip

Now we formulate the above results into a theorem.
\begin{theorem} \label{the_function_Phi}
	Let $\Omega$ be a bounded simply connected domain such that $\Gamma=\bd \Omega\cap D(\zeta_0, r)$ is a (union of) Jordan arc(s). Also, let $\Phi$ be a (non-trivial) holomorphic function of two variables defined in $D(\zeta_0, r)\times D(\bar\zeta_0, r),$ and suppose there exists a function $R$
	\begin{enumerate}[{\rm(i)}]
		\item holomorphic on $\Omega,$
		\item continuous on $\clos{\Omega},$ and such that
		\item $R(\zeta)=\Phi(\zeta,\0{\zeta})$ for all $\zeta\in \Gamma$.
	\end{enumerate}
	Then, there exists a closed set, $E\subset \Gamma,$ of zero harmonic measure so that $\Gamma\setminus E$ is a countable union of regular real analytic simple arcs except possibly for some cusps. The cusps (if they exist) point into $\Omega$ and may only accumulate on $E$.
\end{theorem}

\section{The \texorpdfstring{$\mathcal{U}$-$\mathcal{V}$}{U-V} problem} \label{sec_u_v} 

In this section, we are interested in the following setup.

\smallskip

Let $\Omega$ be a simply connected open set in ${\C}$ and let $\zeta_0\in\bd \Omega$ be a boundary point of $\Omega$. Assume that for some $\rho>0$ the connected component, $\Gamma$, of $\bd \Omega\cap D(\zeta_0,\rho)$ containing $\zeta_0$ is a Jordan curve. Note that $\rho\geq\dist(\zeta_0,\bd \Omega\setminus \Gamma)>0$. For convenience we will write simply $\Omega$ to denote $\Omega\cap D(\zeta_0,\rho)$.

Let $A$ be an analytic function in a neighbourhood, $D(\zeta_0,\epsilon)$, of $\zeta_0$ and suppose we have two functions $\mathcal{U}$ and $\mathcal{V}$ defined on $\Omega$ that are not proportional and have the following properties:
\begin{enumerate}[I)]
	\item $\mathcal{U}$ and $\mathcal{V}$ are positive and harmonic on $\Omega$,
	\item they are continuous on $\Omega\cup \Gamma$,
	\item $\mathcal{U}=\mathcal{V}=0$ on $\Gamma$, and
	\item $\frac{\mathcal{U}(\zeta)}{\mathcal{V}(\zeta)}=|A(\zeta)|^2\neq{\const}$ for $\zeta\in \Gamma$.
	\label{eq_u_v_boundary_values}
\end{enumerate}
Notice that since $\cU\neq c\cV$, the function $|A|$ needs to be non-constant. Otherwise, we could have $\cU=c\cV$ and all our conditions work trivially for any $\Gamma$. Also, we may assume that $\rho<\epsilon$ without loss of generality (so that $A$ is defined over the whole $\Omega$) to avoid unnecessary technical difficulties.

Formula \eqref{eq_u_v_boundary_values} is to be understood in the sense of limits, i.e., the limit of $\mathcal{U}(z)/\mathcal{V}(z)$ as $\Omega\ni z\to \zeta\in \Gamma$ exists and is equal to $|A(\zeta)|^2$. In fact, this limit always exists when $\Omega$ is simply connected and $\Gamma$ is Jordan (see Remark \ref{u_v_Harnack}), so the only assumption here is the values it takes.

Consider a conformal map from the Poincar\'e plane to $\Omega$, $\phi: \H\to \Omega$. Since~$\Gamma$ is connected and Jordan, Carath{\'e}odory's theorem implies that $\phi$ extends conformally to a function (abusing the notation) $\phi: \H\cup \gamma\to \Omega\cup \Gamma$ which we can pick so that $\gamma\subset{\R}$ is some bounded open interval with $\phi(\gamma)=\Gamma$ and $\phi(0)=\zeta_0$. Utilizing this $\phi$, we can ``transfer'' the information about $\mathcal{U}$ and $\mathcal{V}$ over $\Omega$ to information over $\H$. Define
\[u\equiv\mathcal{U}\comp \phi,\quad v\equiv\mathcal{V}\comp \phi\inline{and}a\equiv A\comp \phi\]
and note that $a$ is analytic on $\H$ and continuous on $\H\cup \gamma$. As above, we have
\begin{enumerate}[i)]
	\item $u$ and $v$ are positive and harmonic on $\H$,
	\item they are continuous on $\H\cup \gamma$,
	\item $u=v=0$ on $\gamma$, and
	\item $\frac{u}{v}=|a|^2$ on $\gamma$. \label{real_analytic_h}
\end{enumerate}
Again, \eqref{real_analytic_h} is to be understood in the sense of limits.

Now, harmonically extend $u$ and $v$ to $\H\cup \gamma\cup\H^-$ by
\[u^*(z)=\begin{cases}
	u(z), & z\in\H\\
	0, & z\in \gamma\\
	-{u(\0 z)}, & z\in\H^-
\end{cases}
\inline{and}
v^*(z)=\begin{cases}
	v(z), & z\in\H\\
	0, & z\in \gamma\\
	-{v(\0 z)}, & z\in\H^-
\end{cases}\]
and let $h$ be the function
\[h(z)=\begin{cases}
	\frac{u^*(z)}{v^*(z)}, & z\in\H\cup\H^-,\\
	\frac{u_y^*(z)}{v_y^*(z)}, & z\in \gamma.
\end{cases}\]

We claim that $h$ is well defined and, in fact, real analytic on $\H\cup \gamma\cup\H^-$. Indeed, using Harnack's inequality, for any $(x,0)\in \gamma$ there exists a constant $c>0$ (dependent on $v^*$) such that
\begin{align} \label{eq_Harnack}
	c\frac{y}{2-y}\leq v^*(x,y)\leq c\frac{2-y}{y} & \quad\text{for every}\ 0<y<1,\ \text{or}\notag\\
	c\frac{1}{2-y}\leq\frac{v^*(x,y)}{y}\leq c\frac{2-y}{y^2}. &
\end{align}
Recall that $v^*(x,0)=0$ and take limits as $y\to 0^+$. Since $v^*$ is harmonic on $\H\cup \gamma\cup\H^-$, \eqref{eq_Harnack} guarantees that $v^*_y>0$ on $\gamma$ (the same holds true for $u^*$) and therefore the limit
\[\lim_{y\to 0}\frac{u^*(x,y)}{v^*(x,y)}=\frac{u^*_y(x,0)}{v^*_y(x,0)}\]
exists and is finite. Hence, $h$ is a well-defined continuous function on $\H\cup \gamma\cup\H^-$. In fact, since $u^*_y$ and $v^*_y$ are real analytic and non-zero around $\gamma$, $h$ is also real analytic on $\H\cup \gamma\cup\H^-$. What is more is that
\begin{equation} \label{eq_h_on_gamma}
	h(\xi)=\frac{u^*_y(\xi)}{v^*_y(\xi)}=\lim_{\H\ni z\to \xi}\frac{u(z)}{v(z)}=|a(\xi)|^2\afterline{for any}\xi\in \gamma
\end{equation}
because of \eqref{real_analytic_h} and therefore $|a|^2$ is also real analytic on $\gamma$.

\begin{rem} \label{u_v_Harnack}
	The above is the reason why relation \eqref{eq_u_v_boundary_values} is meaningful. When we write $\frac{\mathcal{U}}{\mathcal{V}}$ on $\Gamma$, it really means the limit of $h\comp \phi^{-1}$ as we approach~$\Gamma$ from the inside of $\Omega$. This limit always exist on a Jordan arc $\Gamma$ when $\Omega$ is simply connected thanks to Harnack's inequality.
	
	It is worth mentioning the work of Jerison and Kenig who showed \cite[Theorems 5.1 and 7.9]{JerKen1982} that equation \eqref{eq_u_v_boundary_values} makes sense whenever $\Omega$ is assumed to be a non-tangentially accessible (NTA) domain.
\end{rem}

\smallskip

Next, consider $h\2{\gamma}$. Its power series around $0\in \gamma$ is given by
\[h\2{\gamma}(x)=\sum_{n=0}^\infty b_nx^n\]
for some real numbers $b_0,b_1,\dots$ This readily extends to a complex analytic function, say $r$, on some open neighbourhood, $D(0,\epsilon')$:
\[r(z)=\sum_{n=0}^\infty b_nz^n,\]
where we can choose $\phi$ and $\epsilon'$ so that $\gamma\subset D(0,\epsilon')$. Of course, by construction and from \eqref{eq_h_on_gamma} we get $r\2{\gamma}=h\2{\gamma}=|a|^2$.

At this point, we want to ``shift'' everything back at $\Omega$. We set 
\[V\equiv \phi(\H\cap D(0,\epsilon'))\subset \Omega\]
and observe that $\bd V$ is a closed Jordan arc such that $\Gamma\varsubsetneq\bd V\cap D(\zeta_0,\rho)$. Define a new function
\begin{equation} \label{eq_the_function_R}
	R\equiv r\comp (\phi^{-1})\2{V\cup \Gamma},
\end{equation}
which is holomorphic on $V$, continuous on $V\cup \Gamma$, and on $\Gamma$ it satisfies $R(\zeta)=|A(\zeta)|^2$.

\smallskip

Now, consider the function $\Phi(z,w)=A(z)\0{A(\0{w})}$. $\Phi$ is holomorphic on $D(\zeta_0,\epsilon)\xx D(\0{\zeta}_0,\epsilon)$ and it satisfies $\Phi(\zeta,\0{\zeta})=A(\zeta)\0{A(\zeta)}=|A(\zeta)|^2$ when $z=\0 w=\zeta\in \Gamma$. As a corollary to Theorem~\ref{the_function_Phi}, the next theorem follows.

\begin{theorem} \label{U_V_side_theorem}
	Let $\Omega$ be a bounded simply connected domain in ${\C}$ and let $\Gamma$ be an open Jordan arc of its boundary with $\zeta_0\in \Gamma$. Suppose there are two positive non-proportional harmonic functions $\mathcal{U}$ and $\mathcal{V}$ on $\Omega$ continuous on $\Omega\cup \Gamma$ and such that
	\[\mathcal{U}(\zeta)=\mathcal{V}(\zeta)=0\inline{and}\frac{\mathcal{U}(\zeta)}{\mathcal{V}(\zeta)}=|A(\zeta)|^2\afterline{for all}\zeta\in \Gamma,\]
	where $A$ is a non-trivial analytic function on a neighbourhood of $\Omega$.
	
	Then, there exists some neighbourhood $D$ of $\zeta_0$ and a closed set $E\subset \Gamma$ of zero harmonic measure so that $(\Gamma\setminus E)\cap D$ is a countable union of regular real analytic simple arcs except possibly for some cusps. The cusps (if they exist) point into $\Omega$ and may only accumulate on $E\cap D$.
\end{theorem}

Of course, Theorems~\ref{the_function_Phi} and~\ref{U_V_side_theorem} are somewhat far from Sakai's result. Nevertheless, because of the special form of the function $\Phi(z,w)=A(z)\0{A(\0 w)}$, we can actually say more in this case.

\begin{prop} \label{U_V_auxiliary_proposition}
	Let $\Omega$ be a bounded simply connected domain in ${\C}$ and let~$\Gamma$ be an open Jordan arc of its boundary with $\zeta_0\in \Gamma$. Suppose there are two positive non-proportional harmonic functions $\mathcal{U}$ and $\mathcal{V}$ on $\Omega$ continuous on $\Omega\cup \Gamma$ and such that
	\[\mathcal{U}(\zeta)=\mathcal{V}(\zeta)=0\inline{and}\frac{\mathcal{U}(\zeta)}{\mathcal{V}(\zeta)}=|A(\zeta)|^2\afterline{for all}\zeta\in \Gamma\]
	where $A$ is a non-trivial analytic function on a neighbourhood of $\Gamma$.
	
	Then, there exists a neighbourhood $D$ of $\zeta_0$ and a function $R$ satisfying the following:
	\begin{enumerate}[{\rm(i)}]
		\item $R$ is holomorphic on $\Omega\cap D,$
		\item $R$ is continuous on $(\Omega\cup \Gamma)\cap D$ and
		\item $R(\zeta)=|A(\zeta)|^2$ for $\zeta\in \Gamma\cap D$.
	\end{enumerate}
	Additionally, for any $\zeta_0\in \Gamma$ with $A'(\zeta_0)\neq0$ either
	\begin{enumerate}[{\rm(1)}]
		\item there exist a function $\Psi_1$ holomorphic and univalent on $\Omega\cap D$ such that $\Psi_1$ is continuous on $(\Omega\cup \Gamma)\cap D,$ and $\Psi_1(\zeta)=|A(\zeta)-A(\zeta_0)|^2$ for $\zeta\in \Gamma\cap D,$~or \label{index_1_case_for_R}
		\item there exist a function $\Psi_2$ holomorphic and univalent on $\Omega\cap D$ such that $\Psi_2^2$ is continuous on $(\Omega\cup \Gamma)\cap D,$ and $\Psi_2^2(\zeta)=|A(\zeta)-A(\zeta_0)|^2$ for $\zeta\in \Gamma\cap D$. \label{index_2_case_for_R}
	\end{enumerate}
\end{prop}

\begin{proof}
	We have already established the existence of such a function $R$ in \eqref{eq_the_function_R}.
	
	For the rest, $A'(\zeta_0)\neq0$ and we may assume without loss of generality that $A$ is conformal on a neighbourhood of $\clos{\Omega}$. Recall that $V$ from the definition of $R$ in \eqref{eq_the_function_R} is such that $\bd V$ is Jordan and $\Gamma\varsubsetneq\bd V\cap D(\zeta_0,\rho)$ when $\Omega\subset D(\zeta_0,\rho)$. Since $A$ is continuous and injective on $\clos V$, there exists some small $\delta$, ${0<\delta\leq\rho}$, such that $\bd(A(V))\cap D(\zeta_0',\delta)\subset A(\Gamma)$.
	
	Now, let $A(\zeta_0)=\zeta_0'$, $\Omega'=A(V)\cap D(\zeta_0',\delta)$, and $\Gamma'=\bd \Omega'\cap D(\zeta_0',\delta)$. The function
	\begin{equation} \label{eq_Schwarz_of_U_V}
		S(z)\equiv\frac{1}{z}R\comp A^{-1}(z)
	\end{equation}
	is a Schwarz function of $\Omega'\cup \Gamma'$ in $D(\zeta_0',\delta)$:
	\begin{enumerate}[(i)]
		\item $S$ is holomorphic on $\Omega'$,
		\item it is continuous on $\Omega'\cup \Gamma'$, and
		\item $S(\zeta)=\frac{1}{\zeta}R(A^{-1}(\zeta))=\0{\zeta}$ on $\Gamma'$.
	\end{enumerate}
	Notice that from \eqref{eq_Harnack} the functions $a=A\comp \phi$ and $A$ are always non-zero and thus $S$ is a well-defined holomorphic function, because $0$ cannot be a point of~$\clos{\Omega}'$.
	
	Finally, consider the function $S_t(z)=S(z+\zeta'_0)-\0{\zeta'_0}$, which is a Schwarz function on $(\Omega'-\zeta'_0)\cup (\Gamma'-\zeta'_0)$ at $0$. From Theorem~\ref{SakProposition}, we know that one of the functions $\Phi_1(z)=zS_t(z)$ and $\Phi_2(z)=\sqrt{zS_t(z)}$ is univalent on $(\Omega'-\zeta'_0)\cap D(0,\delta')$ for some $\delta'\leq \delta$. Changing variables to get back to our initial domain $\Omega$, we find that one of the following functions, $\Psi_1$ or $\Psi_2$, has to be univalent on $\Omega\cap D'$:
	\begin{align*}
		\Psi_1(z)&=(A(z)-A(\zeta_0))\left(\frac{R(z)}{A(z)}-\0{A(\zeta_0)}\right)\\
		\intertext{and}
		\Psi_2(z)&=\sqrt{(A(z)-A(\zeta_0))\left(\frac{R(z)}{A(z)}-\0{A(\zeta_0)}\right)}
	\end{align*}
	for $z\in \Omega\cap D'$, where $D'=A^{-1}(D(\zeta'_0,\delta'))$. The rest of the desired properties are obvious.
\end{proof}

In the above proof, $\Gamma'$ is the image of a Jordan arc under the (conformal) map~$A$. Therefore, the existence of a Schwarz function, $S$, along with Theorem~\ref{SakTheorem} imply that $\Gamma'$, and in turn $\Gamma$, satisfy (1) or (2c) of Theorem~\ref{SakTheorem}. Case~(1) corresponds to~\eqref{index_1_case_for_R} of Proposition~\ref{U_V_auxiliary_proposition} and~(2c) to \eqref{index_2_case_for_R}, that is, $\Gamma'$ (respectively,~$\Gamma$) has a cusp if, and only if, the function 
\[\sqrt{z\big(S(z+\zeta'_0)-\0{\zeta'_0}\big)}\]
is univalent on $(\Omega'-\zeta'_0)\cap D(0,\delta')$ (respectively, $\Psi_2$ on $\Omega\cap D$).

As a consequence, we have the following theorem, which is the main result of this section.
\begin{theorem} \label{U_V_main_theorem}
	Let $\Omega$ be a bounded simply connected domain in ${\C}$ and let $\Gamma$ be an open Jordan arc of its boundary with $\zeta_0\in \Gamma$. Suppose there are two positive non-proportional harmonic functions $\mathcal{U}$ and $\mathcal{V}$ on $\Omega$ continuous on $\Omega\cup \Gamma$ and satisfying
	\[\mathcal{U}(\zeta)=\mathcal{V}(\zeta)=0\inline{and}\frac{\mathcal{U}(\zeta)}{\mathcal{V}(\zeta)}=|A(\zeta)|^2\afterline{for all}\zeta\in \Gamma,\]
	where $A$ is a non-trivial analytic function on a neighbourhood of $\Gamma$.
	
	Then, for all but possibly finitely many points $\zeta_0\in \Gamma$ there exists some small neighbourhood $D$ of $\zeta_0$ such that the following holds:
	\begin{equation} \label{eq_Gamma_is_analytic}
		\Gamma\cap D\text{ is a regular real analytic simple arc through }\zeta_0\text{ except possibly a cusp at }\zeta_0.
	\end{equation}
	The finitely many points around which \eqref{eq_Gamma_is_analytic} might fail are the points $\zeta\in \Gamma$ where $A'(\zeta)=0,$ i.e., where $A$ might not be invertible.
	
	There is a cusp at $\zeta_0$ if and only if \eqref{index_2_case_for_R} of Proposition~{\ref{U_V_auxiliary_proposition}} holds true.
\end{theorem}

Of course, one can ask at this point whether it is possible to actually have a cusp. The answer is yes as the next example shows.

\begin{ex}
	Let $\Omega$ be open and $\Gamma=\bd \Omega\cap D(0,\rho)$ (with $\rho\leq1$ sufficiently small) be such that $\Gamma$ has a cusp at $0$ (i.e., $\zeta_0=0$). Then, from Remarks~\ref{Sak_cusp_case}, for some $\eta>0$, there is a holomorphic function $T$ defined on $\{|z|\leq \eta\}$ that maps conformally the closed upper half-disk $K_\eta=\{|z|\leq \eta\such\Im(z)\geq0\}$ into $\Omega\cup \Gamma$ and $\Gamma\cap D\subset T(-\eta,\eta)$ for some small neighbourhood $D$ of $0$. Also, ${T(0)=0}$ with order $2$. By dilating appropriately, we may assume that everything happens in the unit disk, that is, $\eta=1$, $T$ is defined on $\clos\D$ and is univalent on $K_1=\{|z|\leq1\such\Im(z)\geq0\}$, $T(K_1)\subset \Omega\cup \Gamma$, and $\Gamma\cap D(0,\rho)\subset T(-1,1)$.
	
	Next, consider two positive harmonic functions, $u$ and $v$, on the upper half-disk $\D\cup\H$ that are zero on $(-1,1)$. As we saw in the beginning of this section, $u$ and $v$ can be extended on the whole disk 
	and the ratio $u/v$ is a positive analytic function on $(-1,1)$. Therefore, on $(-1,1)$ we can write $u/v=|a|^2$ for some function $a$ holomorphic on $\D$.
	
	Finally, construct the functions
	\[\cU=u\comp T^{-1},\quad\cV=v\comp T^{-1},\inline{and}A=a\comp T^{-1}.\]
	Then, $A$ is holomorphic around the cusp at $0$, and $\cU,\cV$ are positive harmonic functions on $\Omega\cap D(0,\rho)$ and zero on the boundary $\Gamma$. Moreover, $\cU$ and $\cV$ satisfy $\cU/\cV=|A|^2$ on $\Gamma$.
\end{ex}

\section{Some open ``free boundary'' problems in the spirit of Sakai} \label{free} 

All problems treated above are examples of the so-called free boundary problems (non-variational free boundary problems).

We would like to call the attention of the reader to one open question: what can one say for the boundary of a domain $\Omega$ that is not simply connected but admits positive harmonic functions vanishing on its boundary and whose ratio is ``nice'' on that boundary? Finitely connected situations present no difficulties, but what if, for example, $\Gamma$ is a Cantor set and $\Omega=\D\setminus \Gamma$? Suppose we know that the ratio of two positive harmonic (non-proportional) functions $\cU,\cV$ in~$\Omega$ vanishing on the Cantor set $\Gamma$ has a well-defined ratio on $\Gamma$ (this happens for a wide class of Cantor sets $\Gamma$'s, for example for all regular Cantor sets of positive Hausdorff dimension). Suppose this ratio is equal to $|A(\zeta)|^2\neq \const$ for $\zeta\in \Gamma$, where $A$ is a holomorphic function on $\D$. What we can say about the Cantor set~$\Gamma$? The ``desired'' answer is that this is impossible to happen on any Cantor set. 

This type of problems (we may call them ``one-phase free boundary problems'') appear naturally in certain problems of complex dynamics, see, e.g.,~\cite{Vol1993}. If we would know the aforementioned answer (we conjecture that no Cantor set would allow such a triple $(\cU, \cV, A)$), then a long-standing problem about the dimension of harmonic measure on Cantor repellers would be solved.

\smallskip

Another similar one-phase boundary problem concerns functions in ${\R}^n$ for $n>2$. Let $\Omega$ be a bounded domain in ${\R}^n$, $n>2$, and let $\Gamma=\bd \Omega\cap D(x,r)$, where $x\in\bd \Omega$. Again, let $\cU,\cV$ be two positive (non-proportional) harmonic functions in $\Omega$ vanishing continuously on $\Gamma$. If $\Omega$ is assumed to be a Lipschitz domain, then \cite{JerKen1982} claims that $\cU/\cV$ makes sense on $\Gamma$ and is additionally a H{\"o}lder function on $\Gamma$ (boundary Harnack principle). 

Here is a question. Let $R$ be a real analytic function on $D(x,r)$, $x\in \Gamma$, and let $\cU/\cV=R$ on $\Gamma\cap D(x, r)$. Is it true that $\Gamma\cap D(x, r)$ is real analytic, maybe with the exception of some lower dimensional singular set?

\section*{Acknowledgment}

We are grateful to the referees for the valuable remarks that improved our exposition.

\printbibliography

\end{document}